\documentclass[11pt]{amsart}

\usepackage{amsmath}
\usepackage{graphicx}
\usepackage{amssymb,amsfonts}
\usepackage[mathscr]{eucal}

\usepackage{enumerate}
\newenvironment{enumroman}{\begin{enumerate}[\upshape (i)]}
                                                {\end{enumerate}}

\input xy
\xyoption{all}

\newtheorem{prop}[subsection]{Proposition}
\newtheorem{cor}[subsection]{Corollary}
\newtheorem{lemma}[subsection]{Lemma}

\theoremstyle{definition}
\newtheorem{rem}[subsection]{Remark}
\newtheorem{example}[subsection]{Example}

\newtheorem{definition}[subsection]{Definition}

\newtheorem{term}[subsection]{Terminology}
\newtheorem*{acknowledgements}{Acknowledgements}

\newcommand{\colim}{\text{colim}}

\newcommand{\SSets}{\mathcal{SS}ets}
\newcommand{\xla}{\xleftarrow}
\newcommand{\set}[2]{{\{\,#1\mid#2\,\}}}
\newcommand{\N}{\mathbb{N}}
\newcommand{\defeq}{\overset{\mathrm{def}}=}
\newcommand{\Psh}{\mathrm{Psh}}
\newcommand{\nd}{\mathrm{nd}}
\newcommand{\dg}{\mathrm{dg}}
\newcommand{\Set}{\mathrm{Set}}
\newcommand{\slice}[2]{(#1\downarrow #2)}
\newcommand{\ob}{{\operatorname{ob}}}

\begin{document}

\title{Reedy categories and the $\Theta$-construction}
\author{Julia E. Bergner}
\author{Charles Rezk}
\date{ \today}
\address{Department of Mathematics \\ University of California \\ Riverside, CA 92521}
\email{bergnerj@member.ams.org}

\address{Department of Mathematics \\
University of Illinois at Urbana-Champaign \\
Urbana, IL}
\email{rezk@math.uiuc.edu}

\keywords{Reedy category, $\Theta$-construction}

\subjclass[2010]{Primary 55U35; Secondary 55U10, 18G20, 18G55}

\thanks{The authors were partially supported by NSF grants DMS-0805951 and DMS-1006054.}

\maketitle

\begin{abstract}
We use the notion of multi-Reedy category to prove that, if $\mathcal C$ is a Reedy category, then $\Theta \mathcal C$ is also a Reedy category.  This result gives a new proof that the categories $\Theta_n$ are Reedy categories.  We then define elegant Reedy categories, for which we prove that the Reedy and injective model structures coincide.
\end{abstract}

\section{Introduction}

In this note, we generalize two known facts about the category $\Delta$, which has the structure of a Reedy category.  The first is that the categories $\Theta_k$, obtained from $\Delta$ via iterations of the $\Theta$ construction, are also Reedy categories.  The second is that, on the category of simplicial presheaves on $\Delta$, or functors $\Delta^{op} \rightarrow \SSets$, the Reedy and injective model structures agree.

For the first generalization, we use the notion of multi-Reedy category to prove that for any Reedy category $\mathcal C$, we get that $\Theta \mathcal C$ is also a Reedy category.  For the second, we give a sufficient condition for the Reedy and injective model structures to coincide; such a Reedy category we call \emph{elegant}.

A Reedy category is defined by two subcategories, the direct and inverse subcategories, and a degree function.  (A precise definition is given in Section \ref{reedy}.)   A consequence of the results of this paper is that the Reedy structure on $\Theta_k$ is characterized by:
\begin{enumerate}
\item A map $\alpha\colon \theta \rightarrow \theta'$ is in $\Theta_k^-$ if
  and only if $F\alpha\colon F\theta\rightarrow F\theta'$ is an epimorphism in
  $\Psh(\Theta_k)$.
\item A map $\alpha\colon \theta\rightarrow \theta'$ is in $\Theta_k^+$ if and
  only if $F\alpha\colon F\theta \rightarrow F\theta'$ is a monomorphism in
  $\Psh(\Theta_k)$.
\item There is a degree function $\deg\colon \ob (\Theta_k)\rightarrow \N$, defined
  inductively by
\[ \deg([m](\theta_1,\dots,\theta_m)) = m+\sum_{i=1}^m \deg(\theta_i). \]
\end{enumerate}
Here, $\Psh(\Theta_k)$ denotes the category of presheaves on $\Theta_k$ and $F$ denotes the Yoneda functor.  In itself, this result is not new; $\Theta_k$ was shown to be a Reedy category by Berger \cite{berger}.

\begin{term}
We note two differences in terms from other work.  First, by ``multicategory" we mean a generalization of a category in which a function has a single input but possibly multiple (or no) outputs.  This notion is dual to the usual definition of multicategory, in which a function has multiple inputs but a single output, equivalently defined as a colored operad.  Perhaps the structure we use would better be called a co-multicategory, but we do not because it would further complicate already cumbersome terminology.

Second, some of the ideas in this work are related to similar ones used by Berger and Moerdijk in \cite{bm}.  For example, their definition of EZ-category is more general than ours, in that some of their examples fit into their framework of generalized Reedy categories.
\end{term}

\section{Reedy categories and multi-Reedy categories} \label{reedy}

\subsection{Presheaf categories}

Given a small category $\mathcal C$, we write $\Psh(\mathcal C)$ for the category of functors $\mathcal C^{op} \rightarrow \Set$.   We write $\Psh(\mathcal C, \mathcal M)$ for
the category of functors $\mathcal C^{op} \rightarrow \mathcal M$, where $\mathcal M$ is any category, and $F_\mathcal C\colon \mathcal C\rightarrow \Psh(\mathcal C)$ for the Yoneda functor, defined by
$(F_\mathcal Cc)(d)=\mathcal C(d,c)$.  When clear from the context, we usually
write $F$ for $F_\mathcal C$.

We use the following terminology.
Given an object $c$ of $\mathcal C$ and a presheaf $X \colon \mathcal C \rightarrow \mathcal Set$, a $c$-\emph{point} of $X$ is an element of
the set $X(c)$.  Given an $c$-point $x\in X(c)$, we write
$\bar{x}\colon Fc\rightarrow X$ for the map which classifies the element in $X(c)$.

\subsection{Reedy categories}

Recall that a \emph{Reedy category} is a small category $\mathcal C$
equipped with two \emph{wide} subcategories (i.e., subcategories with
all objects of $\mathcal C$), denoted $\mathcal C^+$ and $\mathcal
C^-$ and called the \emph{direct} and \emph{inverse} subcategories,
respectively, together with a \emph{degree function} $\deg\colon \ob
(\mathcal C)\rightarrow \N$  such that the following hold.
\begin{enumerate}
\item Every morphism $\alpha$ in $\mathcal C$ admits a unique factorization of the form $\alpha=\alpha^+\alpha^-$, where
  $\alpha^+$ is in $\mathcal C^+$ and $\alpha^-$ is in $\mathcal C^-$.
\item For every morphism $\alpha \colon c\rightarrow d$ in $\mathcal C^+$ we have $\deg(c)\leq \deg(d)$, and for every morphism $\alpha\colon c\rightarrow d$ in $\mathcal C^-$, we have $\deg(c)\geq \deg(d)$.  In either case, equality holds if and only if $\alpha$ is an identity map.
\end{enumerate}
Note that, as a consequence, $\mathcal C^+\cap \mathcal C^-$ consists exactly of the identity maps of all the objects, and that identity maps are the only isomorphisms in $\mathcal C$.  Furthermore, for all objects $c$ of $\mathcal C$, the slice categories $\slice{c}{\mathcal C^-}$ and $\slice{\mathcal C^+}{c}$ have finite-dimensional nerve.

\subsection{Multi-Reedy categories}

Associated to a Reedy category is a structure which looks much like that of a multicategory, which has morphisms with one input object but possibly multiple output objects.

Let $\mathcal C$ be a small category.  For any finite sequence of objects $c,d_1,\dots,d_m$ in $\mathcal C$, with $m\geq0$, define
\[ \mathcal C(c;d_1,\dots,d_m) \defeq \mathcal C(c,d_1)\times \cdots \times \mathcal C(c,d_m). \]
This notation also extends to empty sequences; $\mathcal C(c;)$ denotes a one-point set.  We refer to elements $\alpha=(\alpha_s\colon c\rightarrow d_s)_{s=1,\dots,m}$ as \emph{multimorphisms} of $\mathcal C$, and we sometimes use the notation $\alpha\colon c\rightarrow d_1,\dots,d_m$ for such a multimorphism.
Let $\mathcal C(*)$ denote the \emph{symmetric multicategory} whose objects are those of $\mathcal C$, and whose multimorphisms $c\to d_1,\dots,d_s$ are as
indicated above.  Note that $\mathcal C(c;d)=\mathcal C(c,d)$, and that $\mathcal C$ may be viewed as a subcategory of the multicategory $\mathcal C(*)$.

\begin{definition}
A \emph{multi-Reedy category} is a small category $\mathcal C$ equipped with a wide subcategory $\mathcal C^-\subseteq \mathcal C$, and a wide sub-multicategory $\mathcal C^+(*)\subseteq \mathcal C(*)$, together with a function $\deg\colon \ob (\mathcal C) \rightarrow \N$ such that the following hold:
\begin{enumerate}
\item Every multimorphism
\[ \alpha=(\alpha_s\colon c\rightarrow d_s)_{s=1,\dots,m} \]
in $\mathcal C(*)$ admits a unique factorization of the form $\alpha=\alpha^+\alpha^-$, where $\alpha^-\colon c\rightarrow x$ is a morphism in $\mathcal C^-$ and $\alpha^+\colon x\rightarrow d_1,\dots,d_m$ is a multimorphism in $\mathcal C^+(*)$.
\item For every multimorphism $\alpha\colon c\rightarrow d_1,\dots,d_m$ in $\mathcal C^+(*)$ we have
\[ \deg(c)\leq \sum_{i=1}^m \deg(d_i). \]
If $\alpha\colon c\rightarrow d$ is a morphism in $\mathcal C^+= \mathcal C\cap \mathcal C^+(*)$, then $\deg(c)=\deg(d)$ if and only if $\alpha$ is an identity map.  For every morphism $\alpha\colon c\rightarrow d$ in $\mathcal C^-$, we have $\deg(c)\geq \deg(d)$, with equality if and only if $\alpha$ is an identity map.
\end{enumerate}
\end{definition}

Note that for degree reasons, $\mathcal C^+(c;) =\varnothing$ if $\deg(c)>0$, while $\mathcal C^+(c;)$ is non-empty if $\deg(c)=0$.  In particular, if $c$ is any object in $\mathcal C$, there exists
a unique map $\sigma\colon c\rightarrow c_0$ in $\mathcal C^-$ where $c_0$ is an object of degree $0$.

The proof of the following proposition follows from the above constructions.

\begin{prop}
If $\mathcal C$ is a multi-Reedy category, then $\mathcal C$ is a Reedy category with inverse category $\mathcal C^-$, direct category $\mathcal C^+= \mathcal C^+(*)\cap \mathcal C$, and degree
function $\deg$.
\end{prop}

\begin{example}
The terminal category $\mathcal C=1$, with the subcategory $\mathcal C^-= \mathcal C$ and the sub-multicategory $\mathcal C^+(*)= \mathcal C(*)$, and degree function $\deg\colon \ob(\mathcal C)\rightarrow \N$, defined by $d(1)=0$, is a multi-Reedy category.
\end{example}

\begin{example}
Let $\Delta$ be the skeletal category of non-empty finite totally ordered sets.  Let $\Delta^-\subseteq \Delta$ be the subcategory of $\Delta$ consisting of surjective maps, and let $\Delta^+(*)\subseteq \Delta(*)$ be the submulticategory consisting of sequences of maps
\[ \alpha_s\colon [m] \rightarrow [n_s], \qquad s=1,\dots,u \]
which form a monomorphic family; i.e., if $\beta,\beta'\colon [k]\rightarrow[m]$ satisfy $\alpha_s\beta=\alpha_s\beta'$ for all $s=1,\dots,u$, then $\beta=\beta'$.  Let $\deg\colon \ob (\Delta)\rightarrow \N$ be defined by $\deg([m])=m$.  Then $\Delta$ is a multi-Reedy category.

Note that the set $\Delta^+([m];[n_1],\dots,[n_r])$ corresponds to the set of non-degenerate $m$-simplices in the prism $\Delta^{n_1}\times\cdots \times \Delta^{n_r}$.
\end{example}

\begin{rem}
Note that the notion of multi-Reedy category, while having the structure of a multicategory, is being associated to an ordinary Reedy category.  This definition can be extended to an arbitrary multicategory, thus giving rise to the notion of a ``Reedy multicategory", as we investigate briefly in Section \ref{Reedymulti}.
\end{rem}

\subsection{The $\Theta$ construction}

Given a small category $\mathcal C$, we define $\Theta \mathcal C$ to be the category
whose objects are $[m](c_1,\dots,c_m)$ where $m\geq0$, and $c_1,\dots,c_m\in
  \ob (\mathcal C)$, and such that morphisms
  \[ [m](c_1,\dots,c_m)\rightarrow [n](d_1,\dots,d_n) \]
correspond to $(\alpha,\{f_i\})$, where $\alpha\colon [m]\rightarrow [n]$ is a morphism of $\Delta$, and for each $i=1,\dots,m$,
\[ f_i\colon c_i\rightarrow d_{\delta(i-1)+1},\cdots,d_{\alpha(i)} \]
is a multimorphism in $\mathcal C(*)$, which is to say $f_i=(f_{ij})$ where $f_{ij}\colon c_i\rightarrow d_j$ for $\delta(i-1)<j\leq \delta(i)$ is a morphism of $\mathcal C$.

\subsection{Multi-Reedy categories preserved under applying $\Theta$}

Let $\mathcal C$ be a multi-Reedy category, and consider the category $\Theta \mathcal C$.  We make the following definitions.
\begin{itemize}
\item Let  $(\Theta \mathcal C)^-\subseteq \Theta \mathcal C$ be the collection of morphisms
\[ f=(\alpha,\{f_i\})\colon [m](c_1,\dots,c_m)\rightarrow [n](d_1,\dots,d_n) \]
such that $\alpha\colon [m]\rightarrow [n]$ is in $\Delta^-$, and for each $i=1,\dots, m$ such that $\alpha(i-1)<\alpha(i)$, the map $f_i\colon c_i\rightarrow d_{\alpha(i)}$ is in $\mathcal C^-$.

\item Let $(\Theta \mathcal C)^+(*)\subseteq (\Theta \mathcal C)(*)$ be the collection of
multimorphisms $f=(f_s)_{s=1,\dots,u}$, where
\[ f_s=(\alpha_s,\{f_{si}\})\colon [m](c_1,\dots,c_m)\rightarrow [n_s](d_{s1},\dots, d_{sn_s}) \]
such that the multimap $\alpha=(\alpha_s) \colon [m]\rightarrow [n_1],\dots,[n_u]$ is in $\Delta^+(*)$ and for each $i$, the multimap
\[ (f_{sij}\colon c_i\rightarrow d_{sj})_{s=1,\dots,u,\; j=\alpha_s(i-1)+1,\dots,\alpha_s(i)} \]
is in $\mathcal C^+(*)$.

\item Let $\deg\colon \ob (\Theta \mathcal C)\rightarrow \N$ be defined by
\[ \deg([m](c_1,\dots,c_m)) = m+\sum_{i=1}^m \deg(c_i). \]
\end{itemize}

\begin{prop} \label{theta}
  Let $\mathcal C$ be a multi-Reedy category.  Then $\Theta \mathcal C$ is a multi-Reedy category, with $(\Theta \mathcal C)^-$, $(\Theta \mathcal C)^+(*)$ and $\deg\colon
  \ob(\Theta \mathcal C)\rightarrow \N$ defined as above.  In particular, $\Theta \mathcal C$ admits the structure of a Reedy category.
\end{prop}

We assure the reader that the proof is entirely formal; however, we will do our best to obscure the point by presenting a proof full of tedious multiple subscripts.

\begin{proof}
First we observe that $(\Theta \mathcal C)^-$ is closed under composition and contains identity maps; i.e., it is a subcategory of $\Theta \mathcal C$.  Notice that $(\Theta \mathcal C)^-$ contains all identity maps. Suppose we have two morphisms in $(\Theta \mathcal C)^-$ of the form
\[ [m](c_1,\dots,c_m)\xrightarrow{f=(\sigma,f_i)} [n](d_1,\dots,d_n) \xrightarrow{g=(\tau,g_j)} [p](e_1,\dots,e_p). \]
The composite has the form $h=(\tau\sigma, h_i)$, where $h_i$ is defined exactly if $\tau\sigma(i-1)<\tau\sigma(i)$, in which case $h_i=g_{\sigma(i)}f_i\colon c_i\rightarrow e_{\tau\sigma(i)}$.  Since
$\tau\sigma\in \Delta^-$ and $h_i=g_{\sigma(i)}f_i$ is in $\mathcal C^-$, we see that $h$ is in $(\Theta \mathcal C)^-$.

Next we observe that $(\Theta \mathcal C)^+(*)$ is closed under multi-composition and contains identity maps; i.e., it is a sub-multicategory of $(\Theta \mathcal C)(*)$.  Again, note that $(\Theta \mathcal C)^+(*)$ contains all identity maps. Suppose we have a multimorphism $f$ in $(\Theta \mathcal C)^+(*)$ of the form
\[ f=\bigl(f_s\colon [m](c_1,\dots,c_m) \rightarrow [n_s](d_{s1},\dots,d_{sn_s})\bigr)_{s=1,\dots, u} \]
with $f_s=(\delta_s,f_{si})$ where $\delta_s\colon [m]\rightarrow [n_s]$ and $f_{si}=(f_{sij}\colon c_i\rightarrow d_j)_{\delta_s(i-1)<j\leq \delta_s(i)}$, and suppose we have a sequence of multimorphisms $g_1,\dots,g_u$ in $(\Theta \mathcal C)^-(*)$, with each $g_s$ of the form
\[ g_s=\bigl(g_{st}\colon [n_s](d_{s1},\dots,d_{sn_s})\rightarrow [p_{st}](e_{st1},\dots,e_{stp_{st}}) \bigr)_{t=1,\dots,v_s}, \]
where $g_{st}=(\varepsilon_{st}, g_{stj})$, with $\varepsilon_{st}\colon [n_s]\rightarrow [p_{st}]$ and
\[ g_{stj}=(g_{stjk}\colon d_{sj}\rightarrow e_{stk})_{\varepsilon_{st}(j-1)<k\leq \varepsilon_{st}(j)}. \]
The composite multimorphism
\[ h = \bigl(h_{st}\colon [m](c_1,\dots,c_m)\rightarrow [p_{st}](e_{st1},\dots,e_{stp_{st}})\bigr)_{s=1,\dots,u,\; t=1,\dots,v_u} \]
is such that for each $s=1,\dots, u$ and $t=1,\dots,v_s$, the map $h_{st}=(\varepsilon_{st}\delta_s, h_{sti})$ in $\Theta \mathcal C$ is defined so
that the multimap $h_{sti}$ in $\mathcal C(*)$ is given by
\[ h_{sti}=\bigl( h_{stijk}=g_{stjk}f_{sij}\colon c_i\rightarrow e_{stk} \bigr)_{\delta_s(i-1)<j\leq \delta_s(i),\; \varepsilon_{st}(j-1)<k\leq \varepsilon_{st}(j)}. \]
Since $\Delta^+(*)$ is a sub-multicategory of $\Delta(*)$, we get that that $(\varepsilon_{st}\delta_s\colon [n]\rightarrow [p_{st}])_{st}$ is a multimorphism in $\Delta^+(*)$, while since $\mathcal C^+(*)$ is a sub-multicategory of $\mathcal C(*)$, we have that for each $s$, $t$, and $i$, the multimap $h_{sti}$ is in $\mathcal C^+(*)$.  Thus, the multimap $h$ is in $(\Theta \mathcal C)^+(*)$ as desired.

Next, suppose we are given a multimorphism $f=(f_s)_{s=1,\dots,u}$ in $(\Theta \mathcal C)(*)$, where
\[ f_s = (\alpha_s, f_{si})\colon [m](c_1,\dots,c_m) \rightarrow [p_s](e_{s1},\dots,e_{sp_s}). \]
We will show that there is a unique factorization of $f$ into a morphism $g$ of $(\Theta \mathcal C)^-$ followed by a multimorphism $h$ of $(\Theta \mathcal C)^+(*)$. Since $\alpha=(\alpha_s)_{s=1,\dots,u}$ is a multimorphism in $\Delta(*)$ it admits a \emph{unique} factorization $\alpha=\delta\sigma$, where $\sigma\colon [m]\rightarrow [n]$ is in $\Delta^-$
and
\[ \delta=(\delta_s\colon [n]\rightarrow [p_s])_{s=1,\dots,u} \]
is in $\Delta^+(*)$.  Thus, any factorization $f=hg$ of the kind we want must be such that
\[ g=(\sigma,g_i),\qquad g_i\colon c_i\rightarrow d_{\sigma(i)}\; \text{defined when $\sigma(i-1)<\sigma(i)$}, \]
and $h=(h_s)_{s=1,\dots,u}$ such that
\[ h_s=(\delta_s, h_{sj}),\qquad h_{sj}\colon d_j \rightarrow e_{\delta(j-1)+1},\dots, e_{\delta(j)}, \]
and so that for each $i=1,\dots,m$ such that $\sigma(i-1)<\sigma(i)$, the composite of the morphism $g_i$ of $\mathcal C$ with the multimorphism $h_{*\sigma(i)}=(h_{s\sigma(i)})_{s=1,\dots,u}$ of $\mathcal C(*)$ must be equal to the multimorphism $f_{*i}=(f_{si})_{s=1,\dots,m}$ of $\mathcal C(*)$.  In fact, since $\mathcal C$ is a multi-Reedy category, there is a \emph{unique} way to produce a factorization $f_{*i}=h_{*\sigma(i)}g_i$ with the property that $g_i$ is in $\mathcal C^-$ and $h_{*\sigma(i)}$ is in $\mathcal C^+(*)$.

Suppose that $f=(\sigma,f_i)\colon [m](c_1,\dots,c_m)\rightarrow [n](d_1,\dots,d_n)$ is a morphism in $(\Theta \mathcal C)^-$.  Then
\begin{align*}
  \deg([m](c_1,\dots,c_m)) &= m +\sum_{i=1}^m \deg(c_i) \\
&\geq n+\sum_{j=1}^n \deg(d_j) \\
&=\deg([n](d_1,\dots,d_n)).
\end{align*}
The inequality in the second line follows from the fact that $m\geq n$ since $\sigma\in \Delta^-$, and the fact that for each $j=1,\dots,n$, there is exactly one $i$ such that $\sigma(i-1)<j\leq \sigma(i)$, for which the map $f_i\colon c_i\rightarrow d_j$ in $\mathcal  C^-$, whence $\deg(c_i)\geq \deg(d_j)$.

If equality of degrees hold, then we must have $m=n$, whence $\sigma$ is the identity map of $[m]$, and thus we must have $\deg(c_i)=\deg(d_i)$ for all $i=1,\dots,m$, whence each $f_i$ is the identity map of $c_i$.

Suppose that $f=(f_s)_{s=1,\dots,u}$ is a multimorphism in $(\Theta \mathcal C)^+(*)$, where
\[ f_s=(\delta_s,f_{si})\colon [m](c_1,\dots,c_m) \rightarrow [n_s](d_{s1},\dots,d_{sn_s}). \]
Since $(\delta_s)\in \Delta^+(*)$, we have $m\leq \sum_{s=1}^u n_s$. For each $i=1,\dots,m$, the multimorphism
\[ f_{*i*}=\bigl(f_{sij}\colon c_i\rightarrow d_{sj}\bigr)_{s=1,\dots,u,\; j=\delta_s(i-1)+1,\dots, \delta_s(i)} \]
is in $\mathcal C^+(*)$, and thus
\[ \deg(c_i)\leq \sum_{s=1}^u \sum_{j=\delta_s(i-1)+1}^{\delta_s(i)} \deg(d_{sj}). \]   For each $s=1,\dots,u$ and $j=1,\dots,n_s$, there is \emph{at most} one $i$
such that $\delta_s(i-1)<j\leq \delta_s(i)$.  Thus
\[ \sum_{i=1}^m \deg(c_i) \leq  \sum_{i=1}^m \sum_{s=1}^u \sum_{j=\delta_s(i-1)+1}^{\delta_s(i)} \deg(d_{sj}) \leq \sum_{s=1}^u \sum_{j=1}^{n_s} \deg(d_{sj}), \]
and thus
\begin{align*}
  \deg([m](c_1,\dots,c_m)) &= m+\sum_{i=1}^m \deg(c_i) \\
&\leq \sum_{s=1}^u n_s + \sum_{s=1}^u \sum_{j=1}^{n_s} \deg(d_{sj}) \\
&=\sum_{s=1}^u \deg([n_s](d_1,\dots,d_{n_s})).
\end{align*}

If $u=1$ and if equality of degrees holds, then we must have $m=n$, whence $\delta_1$ is the identity map, and then we must have $\deg(c_i)=\deg(d_i)$ for $i=1,\dots,m$, whence each $f_i$ is an
identity map.
\end{proof}

\begin{rem}
The $\Theta$ construction can be applied to an arbitrary multicategory
$\mathcal M$; when the multicategory $\mathcal M= \mathcal C(*)$ for
some category $\mathcal C$, then the construction specializes to the
one we have used.  Given a suitable notion of ``Reedy multicategory",
it seems that the above proof can be generalized to show that $\Theta
\mathcal M$ is a Reedy multicategory whenever $\mathcal M$ is; we
state this result in Section \ref{Reedymulti}.  These ideas generalize
Angeltveit's work on enriched Reedy categories constructed from
operads \cite{ang}.
\end{rem}

\subsection{The direct sub-multicategory of $\Theta \mathcal C$}

We give a criterion which can be useful for identifying the morphisms of $(\Theta \mathcal C)^+$, and more generally the multimorphisms of $(\Theta \mathcal C)^+(*)$.

Given a multimorphism $f=(f_s\colon c\rightarrow d_s)_{s=1,\dots,u}$ in the multicategory $\mathcal C(*)$ associated to a category $\mathcal C$, let $Ff$ denote the
induced map of of $\mathcal C$-presheaves
\[ (Ff_1,\dots,Ff_u) \colon Fc\rightarrow Fd_1\times\cdots\times Fd_u. \]

\begin{prop}
Let $\mathcal C$ be a multi-Reedy category, and suppose that for every $f$ in $\mathcal C^+(*)$, the map $Ff$ is a monomorphism in $\Psh(\mathcal C)$.  Then for every $g$ in $(\Theta \mathcal C)^+(*)$, the map $Fg$ is a monomorphism in $\Psh(\Theta \mathcal C)$.
\end{prop}

\begin{proof}
Let $g=(g_s)_{s=1,\dots,u}$ be a multimorphism in $(\Theta \mathcal C)^+(*)$, where
\[g_s=(\beta_s,g_{sj})\colon [n](d_1,\dots,d_n) \rightarrow [p_s](e_{s1},\dots,e_{sp_s}).\]
We need to show that if
\[ f,f'\colon [m](c_1,\dots,c_m)\rightarrow [n](d_1,\dots,d_n) \] are maps in $\Theta \mathcal C$ such that $g_sf=g_sf'$ for all $s=1,\dots,u$, then $f=f'$.   Write $f=(\alpha,f_i)$ and $f'=(\alpha',f_i')$.  Then $g_sf=g_sf'$ implies $\beta_s\alpha=\beta_s\alpha'$ for all $s$, whence $\alpha=\alpha'$ since
\[ (F\beta_s)\colon F[n]\rightarrow F[p_1]\times\cdots \times F[p_u] \] is a monomorphism in $\Psh(\Delta)$.  Thus for each $i=1,\dots,m$ and $\alpha(i-1)<j\leq \alpha(i)$ we have $f_i,f_i'\colon c_i\rightarrow d_j$, which satisfy $g_{sj}f_i=g_{sj}f_i'$ for all $s=1,\dots,u$.  By hypothesis on $\mathcal C$, it follows that $f_i=f_i'$.
\end{proof}

\section{Elegant Reedy categories}

In this section, we give sufficient conditions on a Reedy category to ensure that the Reedy and injective model structures agree.  The categories of degeneracies and inclusions considered by Baues in \cite{baues} are similar.

\subsection{Degenerate and non-degenerate points}

Let $\mathcal C$ be a Reedy category, and suppose that $X$ is an object of $\Psh(\mathcal C)$.

\begin{definition}
A $c$-point $x\in X(c)$ is \emph{degenerate} if there exist $\alpha\colon c\rightarrow d$ in $\mathcal C^-$ and $y\in X(d)$ such that
\begin{enumerate}
\item $(X\alpha)(y)=x$, and

\item  $\alpha$ is not an identity map (or equivalently, $\deg(c)>\deg(d)$).
\end{enumerate}
A $c$-point $x\in X(c)$ is \emph{non-degenerate} if it is not degenerate.
\end{definition}

We write $X_\dg(c), X_\nd(c)\subseteq X(c)$ for the subsets of degenerate and non-degenerate $c$-points of $X$, respectively; thus
\[ X(c)=X_\dg(c)\amalg X_\nd(c). \]  If $f\colon X\rightarrow Y$ in $\Psh(\mathcal C)$ is a map, then $f(X_\dg(c))\subseteq Y_\dg(c)$, while $f^{-1}(Y_\nd(c))\subseteq X_\nd(c)$.

\begin{definition} \label{deg}
A $c$-point $x\in X(c)$ is a \emph{degeneracy} of $y\in X(d)$ if there
exists $\alpha\colon c\rightarrow d$ in $\mathcal C^-$ such that
$x=X(\alpha)(y)$.  Thus, every point is a degeneracy of itself;
a point is non-degenerate if and only it is a degeneracy of
\emph{only} itself.
\end{definition}

Because the slice category $\slice{c}{\mathcal C^-}$ is finite dimensional, every point in $X$ is the degeneracy of at least one non-degenerate point.

For an object $c$ in $\mathcal C$, a point $\alpha\in (Fc)(d)$ is non-degenerate if and only if $\alpha\colon c\rightarrow d$ is in $\mathcal C^+$.  \emph{Warning:}  It is not the case that $\alpha\colon c\rightarrow d$ in $\mathcal C^+$ implies that $F\alpha\colon Fc\rightarrow Fd$ is injective.

\subsection{Elegant Reedy categories}

\begin{definition}
A Reedy category $\mathcal C$ is \emph{elegant} if
\begin{enumerate}
\item [(E)] for every presheaf $X$ in $\Psh(\mathcal C)$, every object
  $c$ in $\mathcal C$, and every $c$-point $x\in X(c)$, there exists a
  unique pair $(\sigma\colon c\rightarrow d$ in $\mathcal C^{-}$ and $y\in
  X_\nd(d))$ such that $(X\sigma)(y)=x$.
\end{enumerate}
\end{definition}

In other words, elegant Reedy categories have the feature that every
point of every presheaf is \emph{uniquely} a degeneracy of
\emph{unique} non-degenerate point.  The standard example of an
elegant Reedy category is the category $\Delta$, as we will see in the
next section.

Condition (E) admits the following equivalent reformulation.
\begin{enumerate}
\item [(E')] For every presheaf $X$ in $\Psh(\mathcal C)$ and every object $c$ in $\mathcal C$, the map
  \begin{align*}
    \coprod_{d\in \ob (\mathcal C)} \coprod_{x\in X_\nd(d)} \mathcal C^-(c,d) & \rightarrow X(c), \\
 (d,x,\alpha) & \mapsto (X\alpha)(x)
  \end{align*}
is a bijection.
\end{enumerate}

\subsection{Characterization of elegant Reedy categories}

The material in this section is prefigured in Gabriel-Zisman \cite[\S II.3]{gz}.

\begin{definition}
A \emph{strong pushout} in a category $\mathcal C$ is a commutative square in $\mathcal C$ such that its image under the Yoneda functor $F\colon \mathcal C\rightarrow \Psh(\mathcal C)$ is a pushout square.
\end{definition}

Note that every strong pushout is actually a pushout in $\mathcal C$.

\begin{prop}\label{prop:elegant-iff-sp}
Let $\mathcal C$ be a Reedy category.  Then $\mathcal C$ is elegant if and only if the following property (SP) holds.
\begin{enumerate}
\item [(SP)] Every pair of maps $\sigma_s\colon c\rightarrow d_s$, $s=1,2$, in $\mathcal C^-$, extends to a commutative square in $\mathcal C^-$ which is a strong pushout in $\mathcal C$.  That is, there exist $\tau_s\colon d_s\rightarrow e$ in $\mathcal C^-$ such that $\tau_1\sigma_1=\tau_2\sigma_2$ and such that
\[\xymatrix{ {Fc} \ar[r]^{F\sigma_1} \ar[d]_{F\sigma_2} & {Fd_1} \ar[d]^{F\tau_1} \\
{Fd_2} \ar[r]_{F\tau_2} & {Fe} }\]
in a pushout square in $\Psh(\mathcal C)$.
\end{enumerate}
\end{prop}

We note some immediate consequences of property (SP).
\begin{enumerate}
\item  In a Reedy category, all isomorphisms are identity maps, and thus colimits are \emph{unique up to identity} if they exist.  Thus, the strong pushout guaranteed by property (SP) is unique up to identity.

\item If $\sigma\colon c\rightarrow d$ is in $\mathcal C^-$, then $F\sigma\colon Fc\rightarrow Fd$ is a \emph{surjective} map of presheaves.  That is,
\[ \colim(Fd\xla{F\sigma} Fc \xrightarrow{F\sigma}Fd) \xrightarrow{(F1_d,F1_d)} Fd \]
is an isomorphism. By condition (SP), there are maps $\tau_s\colon j\rightarrow k$ for $s=1,2$ such that $\tau_1\sigma=\tau_2\sigma$ fitting into a strong pushout square.  Then there is a unique $\gamma\colon e\rightarrow d$ in $\mathcal C$ making the diagram
\[\xymatrix{ {Fc} \ar[r]^{F\sigma} \ar[d]_{F\sigma} & {Fd} \ar[d]_{F\tau_1} \ar[ddr]^{F1_d} \\
{Fd} \ar[r]^{F\tau_2} \ar[rrd]_{F1_d} & {Fe} \ar@{.>}[dr]|{F\gamma} \\
&& {Fd} }\]
commute. Since $\gamma\tau_s=1_d$ and $\tau_s\in \mathcal C^-$ for $s=1,2$, we must have that $\gamma$ is an identity map, since $\mathcal C$ is a Reedy category.

\item The preceding remark implies that each $\sigma\colon c\rightarrow d$ in $\mathcal C^-$ is a \emph{split epimorphism}.  That is, $\sigma$ is a map such that there exists $\delta\colon d\rightarrow
c$ in $\mathcal C$ such that $\sigma\delta=1_d$. Furthermore, a morphism $\alpha\colon c\rightarrow d$ in $\mathcal C$ is in $\mathcal C^-$ if and only if $F\alpha$ is surjective; to prove the if part, note that any split epimorphism in $\mathcal C$ is necessarily in $\mathcal C^-$.

\item The slice category $\slice{c}{\mathcal C^-}$ is cocomplete.  Since all morphisms are epimorphisms, $\slice{c}{\mathcal C^-}$ is a poset.  It has an initial object $1_c\colon c\rightarrow c$, and has finite coproducts by property (SP), and so has a finite colimits.  Since $\slice{c}{\mathcal C^-}$ has finite dimensional nerve, it trivially has all filtered colimits.
\end{enumerate}

To prove the proposition, we use the following lemma, suggested by the
referee.
\begin{lemma}\label{lemma:reedy-split-idempotent}
All idempotents in a Reedy category are split, and thus all retracts
of representable presheaves are representable.  In particular, if $\mathcal C$ is
a Reedy category and
$\epsilon\colon c\rightarrow c$ is such that $\epsilon\epsilon=\epsilon$, then
there exists
$\sigma\colon d\rightarrow c$ in  $\mathcal C^+$ and $\rho\colon c\rightarrow d$ in $\mathcal C^-$ such that
$1_d=\rho\sigma$ and $\epsilon=\sigma\rho$; if $X$ is a retract of $Fc$ with
associated idempotent $F\epsilon$, then $X\cong Fd$.
\end{lemma}

\begin{proof}
First, we factor $\epsilon=\sigma\rho$ with $\rho\colon c\rightarrow d$ in $\mathcal C^-$ and
$\sigma\colon d\rightarrow c$ in $\mathcal C^+$.  Then we
factor $\rho\sigma=\sigma'\rho'$ with $\rho'$ in $\mathcal C^-$ and $\sigma'$ in
$\mathcal C^+$.  By the unique factorization property of Reedy categories, the identity
$\sigma\rho=\sigma\rho\sigma\rho = \sigma\sigma'\rho'\rho$ implies that
$\rho'$ and $\sigma'$ are identity maps, whence $\rho\sigma=1_d$ as
desired.  The statement about retracts follows easily.
\end{proof}

\begin{proof}[Proof of Proposition \ref{prop:elegant-iff-sp}]
Suppose $\mathcal C$ is a Reedy category which satisfies property
(SP).
To prove (E), suppose that $x\in X(c)$, and suppose that we are given
$\sigma_s\colon c\rightarrow d_s$ in $\mathcal C^-$ and $y_s\in X_\nd(d_s)$
for $s=1,2$, such that $(X\sigma_s)(y_s)=x$.  Then there is a unique
dotted arrow $\bar{z}$ making the diagram
\[\xymatrix{ & {Fd_1} \ar[dr]_{F\tau_1} \ar@/^1ex/[drr]^{\bar{y}_1} \\
{Fc} \ar[ur]^{F\sigma_1} \ar[dr]_{F\sigma_2} && {Fe}
\ar@{.>}[r]|{\bar{z}} & {X} \\
& {Fd_2} \ar[ur]^{F\tau_2} \ar@/_1ex/[urr]_{\bar{y}_2} }\]
commute, where the pair of maps $\tau_s\colon d_s\rightarrow e$ in $\mathcal
C^-$ forms the strong pushout in $\mathcal C$ of the original pair of
maps $\sigma_s$.  But since $y_1$ and $y_2$ are non-degenerate, we
must have that $\tau_1$ and $\tau_2$ are identity maps, whence
$y_1=y_2$ and $\sigma_1=\sigma_2$.

Next we prove that if $\mathcal C$ is an elegant Reedy category, then
property (SP) holds.  Suppose that $\sigma_s\colon c\rightarrow d_s$, $s=1,2$,
is a pair of maps in $\mathcal C^-$.  Let $X$ denote the pushout of
$F\sigma_1$ and $F\sigma_2$ in $\Psh(\mathcal C)$, with maps
$\bar{y}_s\colon Fd_s\rightarrow X$ such that
$\bar{y}_1(F\sigma_1)=\bar{y}_2(F\sigma_2)$.  We write $y_s\in X(d_s)$
for the point corresponding to the map $\bar{y_s}$.  Recalling the
fact given immediately after Definition \ref{deg}, there exist
$\tau_s\colon d_s\rightarrow e_s$ in $\mathcal C^-$ and $z_s\in X_\nd(e_s)$
for $s=1,2$ such that $(X\tau_s)(z_s)=y_s$.  Since
$(X\tau_1\sigma_1)(z_1)=(X\tau_2\sigma_2)(z_2)$, the uniqueness
statement of (E) implies that $e_1=e_2$, $z_1=z_2$, and
$\tau_1\sigma_1=\tau_2\sigma_2$.  Write $z=z_1$ and $e=e_1$, and
consider the commutative diagram
\[\xymatrix{ & {Fd_1} \ar[drr]|{F\tau_1} \ar[dr]_{\bar{y}_1}
  \ar@/^1ex/[drrr]^{\bar{y}_1} \\
{Fc} \ar[ur]^{F\sigma_1} \ar[dr]_{F\sigma_2} && {X} \ar@{.>}[r]|-{f} &
{Fe} \ar[r]|{\bar{z}} & {X} \\
& {Fd_2} \ar[urr]|{F\tau_2} \ar[ur]^{\bar{y}_2}
\ar@/_1ex/[urrr]_{\bar{y}_2} }\]
The arrow $f$ exists because $X$ is a pushout, and we have $\bar{z} f
= 1_X$.  Therefore $X$ is a retract of $Fe$, and hence is
representable by \eqref{lemma:reedy-split-idempotent}, and thus
provides the desired strong pushout.
\end{proof}

As a consequence, we obtain the following.
\begin{prop}\label{prop:property-e1}
If $\mathcal C$ is an elegant Reedy category, and if $f\colon X\rightarrow Y$ is a
monomorphism between presheaves of sets on $\mathcal C$, then
$X_\nd(c)=f^{-1}(Y_\nd(c))$ as subsets of $X(c)$.
\end{prop}
\begin{proof}
It is clear that $f^{-1}(Y_\nd(c))\subseteq X_\nd(c)$ from the
definition of non-degeneracy.    To show that $X_\nd(c)\subseteq
f^{-1}(Y_\nd(c))$, note that if $x\in X_\nd(c)$ but $f(x)\in Y(c)$ is
degenerate, then there exists $\alpha\colon c\rightarrow d$ in $\mathcal{C}^-$
fitting into a commutative  square
\[\xymatrix{
{Fc} \ar[r]^{\bar{x}} \ar[d]_{F\alpha}
& {X} \ar[d]^{f(c)}
\\
{Fd} \ar[r]_{\bar{y}} \ar@{.>}[ur]
& {Y}
}\]
As noted in remark (2) after \eqref{prop:elegant-iff-sp} above,
property (SP) implies that if $\sigma\colon
c\rightarrow d$ is in $\mathcal{C}^-$, then $F\sigma$ is an epimorphism of
presheaves.  Since $f(c)$ is a monomorphism there must exist a dotted
arrow making the diagram commute, contradicting the hypothesis that
$x$ is non-degenerate.
\end{proof}

\subsection{Equivalence of Reedy and injective model structures}

Let $\mathcal C$ be a Reedy category. Given a presheaf $X$ in
$\Psh(\mathcal C, \mathcal M)$ on $\mathcal C$ taking values in some
cocomplete category $\mathcal M$, for each object $c$ in $\mathcal C$
the \emph{latching object} at $c$ is an object $L_cX$ of $\mathcal M$
together with a map $p_c=p_c^X\colon L_cX\rightarrow X(c)$, defined by
\[ L_cX \defeq \colim_{(\alpha\colon c\rightarrow d) \in \partial\slice{c}{\mathcal C}} X(d) \xrightarrow{(X\alpha)} X(c), \]
where $\partial\slice{c}{\mathcal C}$ denotes the full subcategory of
the slice category $\slice{c}{\mathcal C}$ whose objects are morphisms
$\alpha\colon c\rightarrow d$ which are not in $\mathcal C^+$.  It is
straightforward to show that the inclusion functor
$\partial\slice{c}{\mathcal C^-}\rightarrow \partial\slice{c}{\mathcal C}$ is
final, so that the natural map
\[ L_cX \rightarrow \colim_{(\alpha\colon c\rightarrow d) \in \partial\slice{c}{\mathcal C^-}} X(d) \]
is an isomorphism.

In the case that $X$ is a set-valued presheaf (i.e., an object of
$\Psh(\mathcal{C})$), it is clear that for  each object $c$ in
$\mathcal C$ the
map $p_c$ factors through a surjection $q_c=q_c^X\colon L_cX \rightarrow
X_\dg(c)$.

\begin{prop}\label{prop:elegant-means-latching-equals-degenerate}
Let $\mathcal{C}$ be an elegant Reedy category.
Then for each set-valued presheaf $X$ on $\mathcal{C}$ and each
  object $c$ of $\mathcal{C}$, the
  map $q_c^X\colon L_cX\rightarrow X_\dg(c)$ is a bijection.
\end{prop}

\begin{proof}
We have already noticed that $q_c^X$ is surjective, so it suffices to
prove injectivity.  Given $x\in X_\dg(c)$, observe that the preimage
of $x$ in $L_cX$ may be identified with the colimit of the functor
$F_x\colon \partial\slice{c}{\mathcal C^-}\rightarrow \Set$ which sends $\alpha\colon
c\rightarrow d$ in $\partial\slice{c}{\mathcal{C}^-}$ to the set of $y\in
X(d)$ such that $X(\alpha)(y)=x$.  The colimit of $F_x$ is thus isomorphic
to the set of path components of the category $\mathcal{F}_x$, having
\begin{itemize}
\item objects the triples $(d,\alpha,y)$ where $d$ is an object of
  $\mathcal{C}$,  $\alpha\colon c\rightarrow d$
  in $\partial\slice{c}{\mathcal{C}^-}$, and $y\in X(d)$ such that
  $(X\alpha)(y)=x$, and
\item morphisms $(d,\alpha,y)\rightarrow (d',\alpha',y')$ the maps $\beta\colon
  d\rightarrow d'$ in $\mathcal{C}^-$ such that $\alpha'=\beta\alpha$ and
  $(X\beta)(y')=y$.
\end{itemize}
The condition that $\mathcal{C}$ is elegant says that for each $x\in
X_\dg(c)$, the category $\mathcal{F}_x$
has an initial object (namely, the unique nondegenerate point associated to
$x$), and therefore the colimit of $F_x$ (equal to $q_X^{-1}(x)$) must
be a single point, as desired.
\end{proof}

\begin{rem}
In a preprint version of this paper, we made a stronger statement in
place of \eqref{prop:elegant-means-latching-equals-degenerate};
namely, that $\mathcal{C}$ is elegant \emph{if and only if} all  $q_c^X\colon
L_cX\rightarrow X_\dg(c)$ are bijections.  This assertion is presumably not true, since
for $q_c^X$ to be a bijection it suffices for each category
$\mathcal{F}_x$ (as in the proof of
\eqref{prop:elegant-means-latching-equals-degenerate}) to have have
connected nerve, whereas elegance imposes the stronger condition that
each $\mathcal{F}_x$ have an initial object.
\end{rem}

\begin{prop}
Let $\mathcal C$ be an elegant Reedy category.  Then
for every monomorphism $f\colon X\rightarrow Y$ in $\Psh(\mathcal
  C)$, and every object $c$ of $\mathcal C$, the induced
  map
  \[ g_c\colon X(c)\coprod_{L_cX} L_cY \rightarrow Y(c)\]
  is a monomorphism.
\end{prop}

\begin{proof}
Given a map $f\colon X\rightarrow Y$ in $\Psh(\mathcal C)$, and an object $c$
of $\mathcal C$, we consider the following commutative diagram.
\[ \xymatrix{ & {X_\nd(c)\amalg Y_\dg(c)} \ar[d]^{\sim}_{u} \ar[dr]^{g_c''} \\
{X(c)\amalg_{L_cX}L_cY} \ar@{->>}[r]^-{\bar{q}} \ar@/_3ex/[rr]_{g_c} &
{X(c)\amalg_{X_\dg(c)} Y_\dg(c)} \ar[r]^-{g_c'} & {Y(c)} } \]
The map $u$ in the diagram is the evident isomorphism produced using
the disjoint
union decompostion $X(c)\cong X_\nd(c)\amalg X_\dg(c)$.  Recall from
the discussion above that there is a tautological map $p_c^X\colon
L_cX\rightarrow X(c)$, which  factors through a
surjection $q_c^X\colon L_cX\rightarrow X_\dg(c)$.  The map $g_c$ in the
diagram is induced by the tautological maps $p_c^X$ and $p_c^Y$, while
the  map $\bar{q}$ in the
diagram is  induced by the surjections $q_c^X$ and $q_c^X$, and
therefore is itself a surjection.  The map $g_c'$ is the unique one
such that $g_c'\bar{q}=g_c$, and $g_c''= g_c'u$.

We first prove that (1) implies (2), i.e., if $\mathcal C$ is elegant
and $f$ is a monomorphism, then the map $g_c$ is a monomorphism.
Given an injective map $f\colon X\rightarrow Y$, and an object $c$ in
$\mathcal C$, Lemma \ref{prop:elegant-means-latching-equals-degenerate} implies
that $q_c^X$ and $q_c^Y$ are isomorphisms, and hence the map $\bar{q}$
in the above diagram is an isomorphism.  Therefore, it will suffice to
show that $g_c''$ is a monomorphism.  The restriction
$g_c''|_{Y_\dg(c)}$ is the inclusion of $Y_\dg(c)$ in $Y(c)$, and so
is injective.  The restriction $g_c''|_{X_\nd(c)}$ is equal to
$f|_{X_\nd(c)}$.  Thus, to show that $g_c''$ is injective it suffices
to show that
\begin{enumroman}
\item $f|_{X_\nd(c)}$ is injective, and

\item $f(X_\nd(c))\subseteq Y_\nd(c)$.
\end{enumroman}
Statement (i) follows since $f$ is injective, and statement
(ii) is \eqref{prop:property-e1}.  Thus $g_c''$ is injective, and thus
$g_c$ is injective.

Next we show that (2) implies (1). If $X=\varnothing$ and we consider a map $f\colon \varnothing \rightarrow Y$ and an object $c$ in $\mathcal C$, then condition (2) implies that $g_c\colon L_cY\rightarrow Y(c)$ is injective, which implies that the surjection $\bar{q}=q_c\colon L_cY\rightarrow Y_\dg(c)$ is actually an isomorphism.  Thus we have proved (E').
\end{proof}

\begin{prop}
Let $\mathcal C$ be an elegant Reedy category, and let $\mathcal
M=\Psh(\mathcal D, \SSets)$ be the category of simplicial set-valued
presheaves on $\mathcal D$, equipped with the injective  model
category structure
(in which  cofibrations are precisely the monomorphisms, and weak
equivalences are pointwise).  Then the
injective and Reedy model structures on $\Psh(\mathcal C, \mathcal M)$
coincide.
\end{prop}

\begin{proof}
Reedy cofibrations are always monomorphisms in $\Psh(\mathcal C,
\mathcal M)$ by \cite[15.7.2]{hirsch}.  The converse statement, that
monomorphisms in $\Psh(\mathcal C, \mathcal M)$ are Reedy
cofibrations, was proved for elegant Reedy categories $\mathcal C$ in
the previous proposition.
\end{proof}

\section{The Eilenberg-Zilber Lemma}

We describe a way to prove that certain Reedy categories are elegant,
using an observation of Eilenberg and Zilber.  The notion of
``EZ-Reedy category'' has also been described in
\cite{bm} and \cite{isaacson-cubical},
with a somewhat different formulation.

Let $\mathcal C$ be a Reedy category.  Given a map $\alpha\colon c\rightarrow d$ in $\mathcal C$, let $\Gamma(\alpha)$ denote the set of \emph{sections} of $\alpha$; that is,
\[ \Gamma(\alpha) = \set{\beta\colon d\rightarrow c \text{ in } \mathcal C}{\alpha\beta=1_d}. \]
Note that if $\sigma\colon c\rightarrow d$ is a map in $\mathcal C^-$, then $\Gamma(\sigma)\subseteq \mathcal C^+(d,c)$.

\begin{definition}
A Reedy category $\mathcal C$ is an \emph{EZ-Reedy category} if the following two conditions hold.
\begin{enumerate}
\item [(EZ1)] For all $\sigma\colon c\rightarrow d$ in $\mathcal C^-$, the set $\Gamma(\sigma)$ is non-empty.
\item [(EZ2)] For all pairs of maps $\sigma,\sigma'\colon c\rightarrow d$ in $\mathcal C^-$, if $\Gamma(\sigma)=\Gamma(\sigma')$, then $\sigma=\sigma'$.
\end{enumerate}
\end{definition}

Note that if $\mathcal C$ is EZ-Reedy, then every $\sigma\colon c\rightarrow d$ in $\mathcal C^-$ is a split epimorphism, and therefore $F\sigma\colon Fc\rightarrow Fd$ is a surjection in $\Psh(\mathcal C)$.

The following argument is due to Eilenberg and Zilber; it is proved in
\cite[\S II.3]{gz}.

\begin{prop}
If $\mathcal C$ is an EZ-Reedy category, then $\mathcal C$ is elegant.
\end{prop}

\begin{proof}
Suppose that $\mathcal C$ is an EZ-Reedy category. Let $f\colon X\rightarrow Y$ be a monomorphism in $\Psh(\mathcal C)$, and suppose that $x\in X(c)$ and $f(x)\in Y_\dg(c)$.  Thus, there exist
$\sigma\colon c\rightarrow d$ in $\mathcal C^-$ and $y\in Y(d)$ such that $(Y\sigma)(y)=f(x)$ and $\sigma$ is not an identity map.  Since $\Gamma(\sigma)$ is non-empty by (EZ1), $\sigma$ is a split epimorphism, and thus $F\sigma$ is a surjection in $\Psh(\mathcal C)$.  Therefore, a dotted arrow exists in the diagram
\[\xymatrix{ {Fc} \ar[r]^{\bar{x}} \ar[d]_{F\sigma} & {X} \ar[d]^{f} \\
{Fd} \ar[r]_{\bar{y}} \ar@{.>}[ur] & {Y} }\]
thus showing that $x\in X(c)$ is also degenerate.  This proves property (E1).

Now suppose $x\in X(c)$, and that there are $\sigma_s\colon c\rightarrow d_s$ in $\mathcal C^-$ and $y_s\in X_\nd(d_s)$ such that $(X\sigma_s)(y_s)=x$, for $s=1,2$.  For any choices of $\delta_s\in \Gamma(\sigma_s)$, we have a diagram
\[\xymatrix{ & {Fd_1}  \ar[dr]^{\bar{y}_1} \ar@/_2ex/[dl]_{F\delta_1} \\
{Fc} \ar[ur]_{F\sigma_1} \ar[dr]^{F\sigma_2} \ar[rr]|{\bar{x}} && {X} \\
& {Fd_2} \ar[ur]_{\bar{y}_2} \ar@/^2ex/[ul]^{F\delta_2} }\]
in which $\bar{y}_1= \bar{y}_1(F\sigma_1)(F\delta_1)= \bar{x}(F\delta_1) = \bar{y}_2(F(\sigma_2\delta_1))$ and $\bar{y}_2=\bar{y}_2(F\sigma_2)(F\delta_2)=\bar{x}(F\delta_2) = \bar{y}_1(F(\sigma_1\delta_2))$.  Since $\bar{y}_1$ and $\bar{y}_2$ are non-degenerate points with a common degeneracy, it follows that $d_1=d_2$, $\bar{y}_1=\bar{y}_2$, and
$\sigma_2\delta_1=1=\sigma_1\delta_2$.  Since $\delta_1$ and $\delta_2$ were arbitrary choices of sections, we see that $\Gamma(\sigma_1)=\Gamma(\sigma_2)$, and thus $\sigma_1=\sigma_2$.
Thus we have proved property (E2).
\end{proof}

\subsection{The category $\Theta_k$ is EZ-Reedy, and so is elegant}


\begin{prop}
Let $\mathcal C$ be a multi-Reedy category such that $\mathcal C(c,d)$ is nonempty for all objects $c,d$ of $\mathcal C$.  If $\mathcal C$ is an EZ-Reedy category, then so is $\Theta \mathcal C$.
\end{prop}

\begin{proof}
Fix a morphism
\[ f=(\sigma,f_i)\colon [m](c_1,\dots,c_m)\rightarrow [n](d_1,\dots,d_n) \]  in $(\Theta \mathcal C)^-$.  We first determine the structure of the set of sections $\Gamma(f)$.  Observe that the
functor $\Theta \mathcal C\rightarrow \Delta$ induces a natural map $\varphi_f\colon \Gamma(f)\rightarrow \Gamma(\sigma)$.  For each $\delta\in \Gamma(\sigma)$, we write $\Gamma_\delta(f)$ for the fiber of $\varphi$ over $\delta$.  It is straightforward to check that $\Gamma_\delta(f)$ consists of all maps of the form $g=(\delta,g_j)$, where each $g_j=(g_{ji}\colon d_j\rightarrow c_i)_{\delta(j-1)<i\leq \delta(j)}$ is a multimorphism in $\mathcal C$, with the following property: for $i$ such that $\sigma(i-1)<j= \sigma(i)$, we have $g_{ji}\in \Gamma(f_i)$.

Thus, $\Gamma_\delta(f)$ is in bijective correspondence with a subset of
\[ \prod_{j=1}^n \prod_{i=\delta(j-1)+1}^{\delta(j)} C(d_j,c_i), \] namely the set $G_\delta(f)=\prod_{j=1}^n \prod_{i=\delta(j-1)+1}^{\delta(j)} G_{ij}(f)$, where
\[ G_{ij}(f) =
\begin{cases}
  \Gamma(f_i) & \text{if $\sigma(i-1)<j= \sigma(i)$,} \\
  \mathcal C(d_j,c_i) & \text{otherwise.}
\end{cases} \]
Because every set $\mathcal C(d_j,c_i)$ is non-empty, and since $\Gamma(f_i)$ is non-empty by hypothesis, we have that $\Gamma_\delta(f)$ is non-empty for each $\delta\in \Gamma(\sigma)$.  Thus $\varphi_f$ is surjective, and since $\Gamma(\sigma)$ is non-empty, proving (EZ1).

Now suppose $f=(\sigma,f_i)$ and $f'=(\sigma',f_i')$ are two maps $[m](c_1,\dots,c_m)\rightarrow [n](d_1,\dots,d_n)$ in $(\Theta \mathcal C)^-$, and that $\Gamma(f)=\Gamma(f')$.  Since $\phi_f\colon \Gamma(f)\rightarrow \Gamma(\sigma)$ and $\phi_{f'}\colon \Gamma(f')\rightarrow \Gamma(\sigma')$ are surjective, we must have $\Gamma(\sigma)=\Gamma(\sigma')$, and thus $\sigma=\sigma'$.  Thus, for each $i$ and $j$ such that $\sigma(i-1)<j=\sigma(i)$, we have maps $f_{ij},f_{ij}'\colon c_i\rightarrow d_j$.  For each $\delta\in \Gamma(\sigma)$ we therefore have $\Gamma_\delta(f)=\Gamma_\delta(f')$, which must therefore both correspond to the same subset of
\[ \prod_{j=1}^n\prod_{\delta(j-1)+1}^{\delta(j)} \mathcal C(d_j,c_i). \]
Therefore $G_{ij}(f)=G_{ij}(f')$ for all $\delta(j-1)<i\leq \delta(j)$.  In particular, for every $i=1,\dots,n$ such that $\sigma(i-1)< \sigma(i)$, we have that $\Gamma(f_i)=G_{i\sigma(i)}(f)=G_{i\sigma(i)}(f') = \Gamma(f_i')$, and hence $f_i=f_i'$ since $\mathcal C$ is EZ-Reedy.
\end{proof}

\begin{cor}
For all $k\geq0$, $\Theta_k$ an elegant Reedy category, with direct and inverse subcategories as described in the introduction.
\end{cor}

\begin{proof}
Proposition \ref{theta} allows us to put a Reedy model category structure on each $\Theta_k$, and we have shown that with this structure, every $\alpha\colon \theta\rightarrow \theta'$ in $\Theta_k^+$ induces a monomorphism $F\alpha\colon F\theta\rightarrow F\theta'$ of presheaves.

Induction on $k$, together with the fact that each $\Theta_k$ has a terminal object, shows that each $\Theta_k$ with the multi-Reedy structure is an EZ-Reedy category. It follows that $\Theta_k$ is elegant, and also that every $\alpha\colon\theta\rightarrow \theta'$ in $\Theta_k^-$ induces an epimorphism $F\alpha\colon F\theta\rightarrow F\theta'$ of presheaves.

Since any map of presheaves factors uniquely, up to isomorphism, into an epimorphism followed by a monomorphism, this argument shows that $\Theta_k^-$ and $\Theta_k^+$ must be exactly the classes of maps
described in the introduction.
\end{proof}

\section{Reedy multicategories} \label{Reedymulti}

With the definition of multi-Reedy category, which is a multicategory arising from a Reedy category, one might ask whether the notion of Reedy category can be extended to that of a Reedy multicategory.  In this section, we propose a definition and give examples.

\begin{definition}
A \emph{Reedy multicategory} is a symmetric multicategory $\mathcal A$ equipped with the following structure: wide submulticategories $\mathcal A^-$ and $\mathcal A^+$, where $\mathcal A^-$  only has multimorphisms with valence at most $1$, and $\mathcal A^+$ only has multimorphisms with valence at least $1$, together with a function $\deg\colon \ob (\mathcal A) \rightarrow \N$ such that the following properties hold.
\begin{enumerate}
\item For all $f\colon a\rightarrow b_1,\dots,b_m$ in $\mathcal A$ with $m\geq1$, there exists a unique factorization in $\mathcal A$ of the form $f=f^+f^-$, where $f^-$ is in $\mathcal A^-$ and $f^+$ is in $\mathcal A^+$.

\item For all $f\colon a\rightarrow b$ in $\mathcal A^-$, we have $\deg(a)\geq \deg(b)$, with equality if and only if $f$ is an identity map.  For all $f\colon a\rightarrow b_1,\dots,b_m$ in $\mathcal A^+$, we have $\deg(a)\leq  \sum_{i=1}^m \deg(b_i)$; when $m=1$, equality holds if and only if $f$ is an identity map.
\end{enumerate}
\end{definition}

Note that the underlying category of $\mathcal A$ is a Reedy category.

\begin{example}
  Let $\mathcal A$ be a symmetric multicategory with one object, i.e., an operad.  Write $\mathcal A_n$ for $\mathcal A(a;a,\dots,a)$, where $a$ is the unique object.  Suppose that $\mathcal A_0=\mathcal A_1=*$.  Then $\mathcal A$ can be given the structure of a Reedy multicategory, with $\mathcal A^-_k=\mathcal A_k$ for $k\leq1$ and $\mathcal A^+_k= \mathcal A_k$ for $k\geq 1$, and with $\deg(a)=0$.
\end{example}

Given a symmetric multicategory $\mathcal A$, we define a category $\Theta \mathcal A$ as follows.  The objects of $\Theta \mathcal A$ are
\[ [m](a_1,\dots,a_m),\qquad m\geq0,\quad a_1,\dots,a_m\in \ob(\mathcal A). \]
The morphisms
\[ f\colon [m](a_1,\dots,a_m) \rightarrow [n](b_1,\dots,b_n) \]
are given by data $(\alpha,\{f_i\})$, where $\alpha\colon [m]\rightarrow [n]$ in $\Delta$, and for each $i=1,\dots,m$, $f_i\colon a_i\rightarrow b_{\delta(i-1)+1},\dots, b_{\delta(i)}$ in $\mathcal A$.

The category $\Theta \mathcal A$ can be extended to a multicategory.  A multimorphism
\[ f\colon [m](a_1,\dots,a_m)\rightarrow [n_1](b_{11},\dots,b_{1n_1}),\dots, [n_u](b_{u1},\dots,b_{un_u}) \]
is given by $(\alpha, \{f_i\})$ consisting of a multimorphism $\alpha\colon [m]\rightarrow [n_1],\dots,[n_u]$ in $\Delta$, i.e., a sequence of morphisms $\alpha_s\colon [m]\rightarrow [n_s]$ in $\Delta$, and for each
$i=1,\dots, m$ a multimorphism
\[ f_i \colon a_i \rightarrow b_{1,\alpha_1(i-1)+1},\dots,b_{1,\alpha_1(i)},\dots,b_{u,\alpha_u(i-1)+1},\dots, b_{u,\alpha_u(i)}; \]
i.e., a multimorphism with target $(b_{sj})_{1\leq s\leq u,\; \delta_s(i-1)<j\leq \delta_s(i)}$.

As remarked above, the same sort of argument used to prove Proposition \ref{theta} can be used to establish the following result.

\begin{prop}
 $\Theta \mathcal A$ is a Reedy multicategory.
\end{prop}

\begin{acknowledgements}
The authors would like to thank the anonymous referee for a careful reading and helpful comments.  The first-named author would also like to thank Clemens Berger for pointing out the reference \cite{baues}.
\end{acknowledgements}

\end{document}